\documentclass{article}%
\usepackage{graphicx}
\usepackage{amsmath}
\usepackage{amsfonts}
\usepackage{amssymb}
\usepackage{setspace}
\usepackage[hmargin=3.5cm,vmargin=3.5cm]{geometry}%
\providecommand{\U}[1]{\protect\rule{.1in}{.1in}}
\providecommand{\U}[1]{\protect\rule{.1in}{.1in}}
\providecommand{\U}[1]{\protect\rule{.1in}{.1in}}
\providecommand{\U}[1]{\protect\rule{.1in}{.1in}}
\providecommand{\U}[1]{\protect\rule{.1in}{.1in}}
\providecommand{\U}[1]{\protect\rule{.1in}{.1in}}
\providecommand{\U}[1]{\protect\rule{.1in}{.1in}}
\providecommand{\U}[1]{\protect\rule{.1in}{.1in}}
\providecommand{\U}[1]{\protect\rule{.1in}{.1in}}
\providecommand{\U}[1]{\protect\rule{.1in}{.1in}}
\providecommand{\U}[1]{\protect\rule{.1in}{.1in}}
\providecommand{\U}[1]{\protect\rule{.1in}{.1in}}
\providecommand{\U}[1]{\protect\rule{.1in}{.1in}}
\providecommand{\U}[1]{\protect\rule{.1in}{.1in}}
\providecommand{\U}[1]{\protect\rule{.1in}{.1in}}
\providecommand{\U}[1]{\protect\rule{.1in}{.1in}}
\providecommand{\U}[1]{\protect\rule{.1in}{.1in}}
\providecommand{\U}[1]{\protect\rule{.1in}{.1in}}
\providecommand{\U}[1]{\protect\rule{.1in}{.1in}}
\providecommand{\U}[1]{\protect\rule{.1in}{.1in}}
\providecommand{\U}[1]{\protect\rule{.1in}{.1in}}
\providecommand{\U}[1]{\protect\rule{.1in}{.1in}}
\providecommand{\U}[1]{\protect\rule{.1in}{.1in}}
\providecommand{\U}[1]{\protect\rule{.1in}{.1in}}
\providecommand{\U}[1]{\protect\rule{.1in}{.1in}}
\providecommand{\U}[1]{\protect\rule{.1in}{.1in}}
\providecommand{\U}[1]{\protect\rule{.1in}{.1in}}
\providecommand{\U}[1]{\protect\rule{.1in}{.1in}}
\providecommand{\U}[1]{\protect\rule{.1in}{.1in}}
\providecommand{\U}[1]{\protect\rule{.1in}{.1in}}
\providecommand{\U}[1]{\protect\rule{.1in}{.1in}}
\providecommand{\U}[1]{\protect\rule{.1in}{.1in}}
\providecommand{\U}[1]{\protect\rule{.1in}{.1in}}
\providecommand{\U}[1]{\protect\rule{.1in}{.1in}}
\providecommand{\U}[1]{\protect\rule{.1in}{.1in}}
\providecommand{\U}[1]{\protect\rule{.1in}{.1in}}
\providecommand{\U}[1]{\protect\rule{.1in}{.1in}}
\providecommand{\U}[1]{\protect\rule{.1in}{.1in}}

\newtheorem{theorem}{Theorem}
{}

\newtheorem{corollary}{Corollary}

\newtheorem{definition}{Definition}

\newtheorem{lemma}{Lemma}
{}

\newtheorem{proposition}{Proposition}
\newtheorem{remark}{Remark}

\newenvironment{proof}[1][Proof]{\textbf{#1.} }{\ \rule{0.5em}{0.5em}}

\begin{document}

\title{Spectral Analysis of the Non-self-adjoint \ Mathieu-Hill Operator }
\author{O. A. Veliev\\{\small Depart. of Math., Dogus University, Ac\i badem, Kadik\"{o}y, \ }\\{\small Istanbul, Turkey.}\ {\small e-mail: oveliev@dogus.edu.tr}}
\date{}
\maketitle

\begin{abstract}
We obtain uniform, with respect to $t$ asymptotic formulas for the eigenvalues
of the operators generated in $[0,1]$ by the Mathieu-Hill equation with a
complex-valued potential and by the $t-$periodic boundary conditions. Then
using it we investigate the non-self-adjoint \ Mathieu-Hill operator$\ H$
generated in $(-\infty,\infty)$ by the same equation and establish the
necessary and sufficient conditions for the potential for which $H$ has no
spectral singularity at infinity and it is an asymptotically spectral operator.

Key Words: Hill operator, Spectral singularities, Spectral operator.

AMS Mathematics Subject Classification: 34L05, 34L20.

\end{abstract}

\section{Introduction}

Let $L(q)$ be the Hill operator generated in $L_{2}(-\infty,\infty)$ by the expression%

\begin{equation}
l(y)=-y^{\prime\prime}+qy,
\end{equation}
where $q$ is a complex-valued summable function on $[0,1]$ and $q(x+1)=q(x)$.
It is well-known that (see [8-10]) the spectrum $S(L(q))$ of the operator
$L(q)$ is the union of the spectra $S(L_{t}(q))$ of the operators $L_{t}(q)$
for $t\in(-\pi,\pi],$ where $L_{t}(q)$ is the operator generated in
$L_{2}[0,1]$ by (1) and the boundary conditions
\begin{equation}
y(1)=e^{it}y(0),\text{ }y^{\prime}(1)=e^{it}y^{\prime}(0).
\end{equation}
The spectrum of $L_{t}(q)$ consist of the eigenvalues that are the roots of
\begin{equation}
F(\lambda)=2\cos t,
\end{equation}
where $F(\lambda)=\varphi^{\prime}(1,\lambda)+\theta(1,\lambda)$, $\varphi$
and $\theta$ are the solutions of the equation $l(y)=\lambda y$ satisfying the
initial conditions $\theta(0,\lambda)=\varphi^{\prime}(0,\lambda)=1$
and$\quad\theta^{\prime}(0,\lambda)=\varphi(0,\lambda)=0.$

The operators $L_{t}(q)$ and $L(q)$ are denoted by $H_{t}$ and $H$
respectively when
\begin{equation}
\text{ }q(x)=ae^{-i2\pi x}+be^{i2\pi x},
\end{equation}
where $a$ and $b$ are the nonzero complex numbers. In the cases $t=0$ and
$t=\pi$ the operator $H_{t}$ was investigated by Djakov and Mitjagin\ [2-5].
In [16] we have found the conditions on the potential (4) such that all
eigenvalues of the periodic, antiperiodic, Dirichlet, and Neumann boundary
value problems are simple. In this paper we consider the operators $H$ and
$H_{t}$ for all values of $t\in(-\pi,\pi].$ First, we obtain the asymptotic
formulas, uniform with respect to $t$ in some intervals, for the eigenvalues
of the operators $H_{t}.$ (Note that, the formula\ $f(k,t)=O(h(k))$ is said to
be uniform with respect to $t$ in a set $I$ if there exist positive constants
$M$ and $N,$ independent of $t,$ such that $\mid f(k,t))\mid<M\mid h(k)\mid$
for all $t\in I$ and $\mid k\mid\geq N).$ Then using these asymptotic
formulas, we investigate the spectral singularities and the asymptotic
spectrality of the operator $H$.

Note that the spectral singularities of the operator $L(q)$ are the points of
its spectrum in neighborhoods of which the projections of $L(q)$ are not
uniformly bounded (see [7] and [12]). McGarvey [9] proved that $L(q)$ is a
spectral operator if and only if the projections of the operators $L_{t}(q)$
are bounded uniformly with respect to $t$ in $(-\pi,\pi]$. Recently, Gesztezy
and Tkachenko [6,7] proved two versions of a criterion for the Hill operator
$L(q)$ with $q\in L_{2}[0,1]$ to be a spectral operator of scalar type, in
sense of Danford, one analytic and one geometric. The analytic version was
stated in term of the solutions of Hill's equation. The geometric version of
the criterion uses algebraic and geometric \ properties of the spectra of
periodic/antiperiodic and Dirichlet boundary value problems.

The problem of describing explicitly, for which potentials $q$ the Hill
operators $L(q)$ are spectral operators appears to have been open for about 50
years. Moreover, the discussed papers show that the set of potentials $q$ for
which $L(q)$ is spectral is a small subset of the periodic functions and it is
very hard to describe explicitly the required subset. In paper [14] we found
the explicit conditions on the potential $q$ such that $L(q)$ is an
asymptotically spectral operator and in [17] we constructed the spectral
expansion for the asymptotically spectral operator. In this paper we find a
criterion for asymptotic spectrality of $H$ stated in term of the potential (4).

The paper consists of 5 sections. In Section 2 we present some preliminary
facts, from [13, 14, 3], which are needed in the following. In Section 3 we
obtain some general results for $L_{t}(q)$ with locally integrable potential
$q.$ In Section 4 using the results of Section 3 we obtain the uniform
asymptotic formulas for the operators $H_{t}$. In Section 5, as a main result
of this paper, we find the\ necessary and sufficient conditions on numbers $a$
and $b$ for which $H$ has no spectral singularity at infinity and it is an
asymptotically spectral operator.

\section{Preliminary Facts}

In this section we present some results of [13, 14, 3] which are used for the
proof of the main results of the paper. We use the following results of [13].

\textbf{Theorem 2 of [13].}\textit{ The eigenvalues }$\lambda_{n}(t)$\textit{
\ and eigenfunctions }$\Psi_{n,t}$\textit{ of the operators }$L_{t}%
(q)$\textit{ for }$t\neq0,\pi,$\textit{ satisfy the following asymptotic
formulas }%
\begin{equation}
\lambda_{n}(t)=(2\pi n+t)^{2}+O(n^{-1}\ln\left\vert n\right\vert ),\text{
}\Psi_{n,t}(x)=e^{i(2\pi n+t)x}+O(n^{-1}).
\end{equation}
\textit{for }$n\rightarrow\infty.$\textit{ For any fixed number }$\rho
\in(0,\pi/2),$\textit{ these asymptotic formulas are uniform with respect to
}$t$\textit{ in }$[\rho,\pi-\rho]$\textit{. Moreover, there exists a positive
number }$N(\rho),$\textit{ independent of }$t,$\textit{ such that the
eigenvalues }$\lambda_{n}(t)$\textit{ for \ }$t\in\lbrack\rho,\pi-\rho
]$\textit{ and}$\mid n\mid>N(\rho)$\textit{ are simple.}

In the paper [14] we obtained the uniform asymptotic formulas in the more
complicated case $t\in\lbrack0,\rho]\cup\lbrack\pi-\rho,\pi],$ when the
potential $q$ satisfies the following conditions:\textit{ }%
\[
q\in W_{1}^{p}[0,1],\mathit{\ }q^{(k)}(0)=q^{(k)}(1),\text{ }q_{n}\sim
q_{-n},\text{ }(q_{n})^{-1}=O(n^{s+1})
\]
for $k=0,1,...,s-1$\textit{ }with some\textit{ }$s\leq p$ and at least one of
the inequalities $\operatorname{Re}q_{n}q_{-n}\geq0$ and $\mid
\operatorname{Im}q_{n}q_{-n}\mid\geq\varepsilon\mid q_{n}q_{-n}\mid$ hold for
some $\varepsilon>0,$ where
\begin{equation}
q_{n}=:(q,e^{i2\pi nx})=:\int_{0}^{1}q(x)e^{-i2\pi nx}dx
\end{equation}
is the Fourier coefficient of $q$ and $a_{n}\sim b_{n}$ means that
$a_{n}=O(b_{n})$ and $b_{n}=O(a_{n}).$ It is clear that these results of
[14]\textbf{ }can not be used for the potential (4). However, we use Remark
2.1 and lot of formulas of [14] that are listed in Remark 1 and as formulas (10)-(25).

\begin{remark}
In Remark 2.1 of [14] \ we proved that here exists a positive integer $N(0)$
such that the disk $U(n,t,\rho)=:\{\lambda\in\mathbb{C}:\left\vert
\lambda-(2\pi n+t)^{2}\right\vert \leq15\pi n\rho\}$ for $t\in\lbrack0,\rho],$
where $15\pi\rho<1,$ and $n>N(0)$ contains two eigenvalues (counting with
multiplicities) denoted by $\lambda_{n,1}(t)$ and $\lambda_{n,2}(t)$ and these
eigenvalues can be chosen as a continuous function of $t$ on the interval
$[0,\rho].$ In addition to these eigenvalues, the operator $L_{t}(q)$ for
$t\in\lbrack0,\rho]$ has only $2N+1$ eigenvalues denoted by $\lambda_{k}(t)$
for $k=0,\pm1,\pm2,...,\pm N$ (see Remark 2.1 of [14]). Similarly, there
exists a positive integer $N(\pi)$ such that the disk $U(n,t,\rho)$ for
$t\in\lbrack\pi-\rho,\pi]$ and $n>N(\pi)$ contains two eigenvalues (counting
with multiplicities) denoted again by $\lambda_{n,1}(t)$ and $\lambda
_{n,2}(t)$ that are continuous function of $t$ on the interval $[\pi-\rho
,\pi].$

Thus for $n>$ $N=:\max\left\{  N(\rho),N(0),N(\pi)\right\}  ,$ the eigenvalues
$\lambda_{n,1}(t)$ and $\lambda_{n,2}(t)$ are continuous on $[0,\rho
]\cup\lbrack\pi-\rho,\pi]$ and for$\mid n\mid>N$ \ \textit{the eigenvalue
}$\lambda_{n}(t),$\textit{ defined by (5), is continuous on \ }$[\rho,\pi
-\rho].$\textit{ Moreover, by Theorem 2 of [13] there exist only two
eigenvalues }$\lambda_{-n}(\rho)$ and $\lambda_{n}(\rho)$ of the operator
$L_{\rho}(q)$ lying in the disk $U(n,\rho,\rho).$ Therefore these 2
eigenvalues coincides with the eigenvalues $\lambda_{n,1}(\rho)$ and
$\lambda_{n,2}(\rho).$ By (5) $\operatorname{Re}(\lambda_{-n}(\rho
))<\operatorname{Re}(\lambda_{n}(\rho)).$ Let $\lambda_{n,2}(\rho))$ be the
eigenvalue whose real part is larger. Then%
\begin{equation}
\lambda_{n,1}(\rho)=\lambda_{-n}(\rho),\text{ }\lambda_{n,2}(t)=\lambda
_{n}(\rho).
\end{equation}
In the same way we obtain that
\begin{equation}
\lambda_{n,1}(\pi-\rho)=\lambda_{n}(\pi-\rho),\text{ }\lambda_{n,2}(\pi
-\rho)=\lambda_{-(n+1)}(\pi-\rho)
\end{equation}
if $\lambda_{n,2}(\pi-\rho))$ is the eigenvalue whose real part is larger. Let
$\Gamma_{-n}$ be the union of the following continuous curves $\left\{
\lambda_{n-1,2}(t):t\in\lbrack\pi-\rho,\pi]\right\}  ,$ $\left\{  \lambda
_{-n}(t):t\in\lbrack\rho,\pi-\rho]\right\}  $ and $\left\{  \lambda
_{n,1}(t):t\in\lbrack0,\rho]\right\}  .$ By (7) and \ (8) these curve are
connected and $\Gamma_{-n}$ is a continuous curve. Similarly, the curve
$\Gamma_{n}$ which is the union of the curves $\left\{  \lambda_{n,2}%
(t):t\in\lbrack0,\rho]\right\}  ,$ $\left\{  \lambda_{n}(t):t\in\lbrack
\rho,\pi-\rho]\right\}  $ and $\left\{  \lambda_{n,1}(t):t\in\lbrack\pi
-\rho,\pi]\right\}  $ is a continuous curve.

Let us redenote $\lambda_{n,1}(t)$ and $\,\lambda_{n,2}(t)$ by $\lambda
_{-n}(t)$ and $\lambda_{n}(t)$ respectively for $n>N$ and $t\in\lbrack
0,\rho].$ Similarly redenote $\lambda_{n,1}(t)$ and $\,\lambda_{n,2}(t)$ by
$\lambda_{n}(t)$ and $\lambda_{-n-1}(t)$ respectively for $n>N$ and
$t\in\lbrack\pi-\rho,\pi].$ In this notation we have $\Gamma_{n}=\{\lambda
_{n}(t):t\in\lbrack0,\pi]\}$\ for $\left\vert n\right\vert >N.$ In this paper
we use both notations: $\lambda_{n}(t)$ and $\lambda_{n,j}(t)$.
\end{remark}

One can readily see that
\begin{equation}
\left\vert \lambda-(2\pi(n-k)+t)^{2}\right\vert >\left\vert k\right\vert
\left\vert 2n-k\right\vert ,\text{ \ }\forall\lambda\in U(n,t,\rho)
\end{equation}
for $k\neq0,2n$ and $t\in\lbrack0,\rho]$, where $n>N$.

In [14] to obtain the uniform, with respect to $t\in\lbrack0,\rho],$
asymptotic formulas for the eigenvalues $\lambda_{n,j}(t)$ we used (9) and the
iteration of the formula
\begin{equation}
(\lambda_{n,j}(t)-(2\pi n+t)^{2})(\Psi_{n,j,t},e^{i(2\pi n+t)x})=(q\Psi
_{n,j,t},e^{i(2\pi n+t)x}),
\end{equation}
where $\Psi_{n,j,t}$ is any normalized eigenfunction corresponding to
$\lambda_{n,j}(t).$ Iterating (10) infinite times we got the following
formula
\begin{equation}
(\lambda_{n,j}(t)-(2\pi n+t)^{2}-A(\lambda_{n,j}(t),t))u_{n,j}(t)=(q_{2n}%
+B(\lambda_{n,j}(t),t))v_{n,j}(t),
\end{equation}
where $u_{n,j}(t)=(\Psi_{n,j,t},e^{i(2\pi n+t)x}),$ $v_{n,j}(t)=(\Psi
_{n,j,t},e^{i(-2\pi n+t)x}),$
\begin{equation}
A(\lambda,t)=\sum_{k=1}^{\infty}a_{k}(\lambda,t),\text{ }B(\lambda
,t)=\sum_{k=1}^{\infty}b_{k}(\lambda,t),
\end{equation}%
\begin{equation}
a_{k}(\lambda,t)=\sum_{n_{1},n_{2},...,n_{k}}q_{-n_{1}-n_{2}-...-n_{k}}%
{\textstyle\prod\limits_{s=1}^{k}}
q_{n_{s}}\left(  \lambda-(2\pi(n-n_{1}-..-n_{s})+t)^{2}\right)  ^{-1},
\end{equation}%
\begin{equation}
b_{k}(\lambda,t)=\sum_{n_{1},n_{2},...,n_{k}}q_{2n-n_{1}-n_{2}-...-n_{k}}%
{\textstyle\prod\limits_{s=1}^{k}}
q_{n_{s}}\left(  \lambda-(2\pi(n-n_{1}-..-n_{s})+t)^{2}\right)  ^{-1}%
\end{equation}
for $\lambda\in U(n,t,\rho)$ (see (37) of [14]).

Similarly, we obtained the formula
\begin{equation}
(\lambda_{n,j}(t)-(-2\pi n+t)^{2}-A^{\prime}(\lambda_{n,j}(t),t))v_{n,j}%
(t)=(q_{-2n}+B^{\prime}(\lambda_{n,j}(t),t))u_{n,j}(t),
\end{equation}
where%
\begin{equation}
A^{\prime}(\lambda,t)=\sum_{k=1}^{\infty}a_{k}^{\prime}(\lambda,t),\text{
}B^{\prime}(\lambda,t)=\sum_{k=1}^{\infty}b_{k}^{\prime}(\lambda,t),
\end{equation}%
\begin{equation}
a_{k}^{\prime}(\lambda,t)=\sum_{n_{1},n_{2},...,n_{k}}q_{-n_{1}-n_{2}%
-...-n_{k}}%
{\textstyle\prod\limits_{s=1}^{k}}
q_{n_{s}}\left(  \lambda-(2\pi(n+n_{1}+..+n_{s})-t)^{2}\right)  ^{-1},
\end{equation}%
\begin{equation}
b_{k}^{\prime}(\lambda,t)=\sum_{n_{1},n_{2},...,n_{k}}q_{-2n-n_{1}%
-n_{2}-...-n_{k}}%
{\textstyle\prod\limits_{s=1}^{k}}
q_{n_{s}}\left(  \lambda-(2\pi(n+n_{1}+..+n_{s})-t)^{2}\right)  ^{-1}%
\end{equation}
for $\lambda\in U(n,t,\rho)$ (see (38) of [14]).

The sums in (13), (14) and (17), (18) are taken under conditions $n_{1}%
+n_{2}+...+n_{s}\neq0,2n$ and $n_{1}+n_{2}+...+n_{s}\neq0,-2n$ respectively,
where $s=1,2,...$

Moreover, it was proved [14] that the equalities
\begin{equation}
a_{k}(\lambda,t),\text{ }b_{k}(\lambda,t),\text{ }a_{k}^{\prime}%
(\lambda,t),\text{ }b_{k}^{\prime}(\lambda,t)=O\left(  (n^{-1}\ln\left\vert
n\right\vert )^{k}\right)
\end{equation}
hold uniformly for $t\in\lbrack0,\rho]$ and $\lambda\in U(n,t,\rho)$\ (see
(34) and (36) of [14]), and derivatives of these functions with respect to
$\lambda$ are $O(n^{-k-1})$ (see the proof of Lemma 2.5) which imply that the
functions $A(\lambda,t),$ $A^{\prime}(\lambda,t),$ $B(\lambda,t)$ and
$B^{\prime}(\lambda,t)$ are analytic on $U(n,t,\rho).$ Moreover, there exists
a constant $K$ such that
\begin{equation}
\mid A(\lambda,t)\mid<Kn^{-1},\mid A^{\prime}(\lambda,t)\mid<Kn^{-1},\mid
B(\lambda,t)\mid<Kn^{-1},\mid B^{\prime}(\lambda,t)\mid<Kn^{-1},
\end{equation}%
\begin{equation}
\mid A(\lambda,t)-A(\mu,t)\mid<Kn^{-2}\mid\lambda-\mu\mid,\mid A^{\prime
}(\lambda,t)-A^{\prime}(\mu,t)\mid<Kn^{-2}\mid\lambda-\mu\mid,
\end{equation}%
\begin{equation}
\mid B(\lambda,t)-B(\mu,t)\mid<Kn^{-2}\mid\lambda-\mu\mid,\mid B^{\prime
}(\lambda,t)-B^{\prime}(\mu,t)\mid<Kn^{-2}\mid\lambda-\mu\mid,
\end{equation}%
\begin{equation}
\mid C(\lambda,t)\mid<tKn^{-1},\text{ }\mid C(\lambda,t))-C(\mu,t))\mid
<tKn^{-2}\mid\lambda-\mu\mid
\end{equation}
for all $\ n>N,$ $t\in\lbrack0,\rho]$ and $\lambda,\mu\in U(n,t,\rho),$ where
$N$ and $U(n,t,\rho)$ are defined in Remark 1, and $C(\lambda,t)=\frac{1}%
{2}(A(\lambda,t)-A^{\prime}(\lambda,t))$ \textbf{(}see Lemma 2.3 and Lemma 2.5
of [14]\textbf{).}

In this paper we use also the following, uniform with respect to $t\in
\lbrack0,\rho],$ equalities from [14] (see (26)-(28) of [14]) for the
normalized eigenfunction $\Psi_{n,j,t}$:
\begin{equation}
\Psi_{n,j,t}(x)=u_{n,j}(t)e^{i(2\pi n+t)x}+v_{n,j}(t)e^{i(-2\pi n+t)x}%
+h_{n,j,t}(x),
\end{equation}%
\begin{equation}
(h_{n,j,t},e^{i(\pm2\pi n+t)x})=0,\text{ }\left\Vert h_{n,j,t}\right\Vert
=O(n^{-1}),\text{ }\left\vert u_{n,j}(t)\right\vert ^{2}+\left\vert
v_{n,j}(t)\right\vert ^{2}=1+O(n^{-2}).
\end{equation}

Besides we use formula (55) of [3] about estimations of $B(\lambda,0)$ and
$B^{\prime}(\lambda,0)$ as follows: \ 

\textit{ Let the potential }$q$\textit{ \ has the form (4), }$\lambda=(2\pi
n)^{2}+z,$\textit{ where \ }$\left\vert z\right\vert <1,$\textit{ and }%
\begin{equation}
p_{n_{1},n_{2},...,n_{k}}(\lambda,0)=q_{2n-n_{1}-n_{2}-...-n_{k}}%
{\textstyle\prod\limits_{s=1}^{k}}
q_{n_{s}}\left(  \lambda-(2\pi(n-n_{1}-..-n_{s}))^{2}\right)  ^{-1}%
\end{equation}
\textit{be summands of }$b_{k}(\lambda,t)$\textit{ for }$t=0$\textit{ (see
(14)). Then using }(55) of [3] \textit{with }$q\geq2$\textit{ and the
estimation}
\[
\sum_{q\geq2}\left(  _{q}^{n+2q}\right)  \left(  \frac{\left\vert
ab\right\vert }{n^{2}}\right)  ^{q}=O(n^{-2})
\]
\textit{of [3] (see the estimation after formula (55) of [3]) and taking into
account that if }$k$\textit{ changes from }$2n+3$\textit{ to }$\infty
,$\textit{ then the number }$q$\textit{ \ of steps}$=-2$\textit{ (that is, in
our notations the number of indices }$n_{1},n_{2},...,n_{k}$\textit{ of (26)
that are equal to }$1$\textit{) changes from }$2$\textit{ to }$\infty
,$\textit{ we obtain}%

\begin{equation}
\sum_{k=2n+3}^{\infty}\sum_{n_{1},n_{2},...,n_{k}}\left\vert p_{n_{1}%
,n_{2},...,n_{k}}(\lambda,0)\right\vert =b_{2n-1}(\lambda,0)O(n^{-2}).
\end{equation}

\section{Some General Results for $L_{t}(q)$ with $q\in L_{1}[0,1]$}

First, in Theorem 1, we consider the cases: $t=0$ and $t=\pi$ which correspond
to the periodic and antiperiodic boundary conditions (see (2)). These cases
were considered by Djakov and Mitjagin in\ [2, 3] and [4, 5] for $q\in
L_{2}[0,1]$ and $q\in H^{-1}[0,1]$ in detail. We obtain similar results by the
methods of our papers [1, 11] for $q\in L_{1}[0,1]$ and use the following
terminology of [11]. If the set of the Jordan chains of $L_{0}(q)$ is
infinite, then we consider the Riesz basis property of the normal system of
eigenfunction and associated function (EAF), defined in [11] as follows. In
the case when the large eigenvalue has geometric multiplicity $2$, we choose
the pair of normalized eigenfunctions so that they are mutually orthogonal. In
the case when only one eigenfunction $\varphi$ corresponds to the double
eigenvalue $\lambda$, we assume that $\left\Vert \varphi\right\Vert =1$ and
choose the associated function to be orthogonal to $\varphi$ (it is uniquely
defined by this condition). If the set of the Jordan chains is finite, then
one do not need to consider the specially chosen normal system of EAF.
Therefore, in this case, instead of the normal system of EAF we use the system
of root functions.

For brevity, we discuss only the periodic problem and denote $\lambda
_{n,j}(t),$ $A(\lambda_{n,j}(t),t)),$ $B(\lambda_{n,j}(t),t)),$ $A^{\prime
}(\lambda_{n,j}(t),t)),$ $B^{\prime}(\lambda_{n,j}(t),t))$ for $t=0$ (see (11)
and (15)) by $\lambda_{n,j},$ $A(\lambda_{n,j}),$ $B(\lambda_{n,j}),$
$A^{\prime}(\lambda_{n,j}),$ $B^{\prime}(\lambda_{n,j})$ respectively. The
antiperiodic problem is similar to the periodic problem. One can readily see
from \ (11), (15), (20) and Remark 1 that
\begin{equation}
\lambda_{n,j}(t)\in d^{-}(r(n),t)\cup d^{+}(r(n),t)\subset U(n,t,\rho),
\end{equation}
for all $\ n>N,$ $t\in\lbrack0,\rho],$ where $r(n)=\max\{\left\vert
q_{2n}\right\vert ,\left\vert q_{-2n}\right\vert \}+2Kn^{-1}$ and $\ d^{\pm
}(r(n),t)$ is the disk with center $(\pm2\pi n+t)^{2}$ and radius $r(n).$
Indeed if $\left\vert u_{n,j}(t)\right\vert \geq\left\vert v_{n,j}%
(t)\right\vert ,$ then using (11) (if $\left\vert v_{n,j}(t)\right\vert
>\left\vert u_{n,j}(t)\right\vert ,$ then using (15)) and (20) we get (28).

In the case $t=0$ the disks $d^{-}(r(n),t)$ and $d^{+}(r(n),t)$ are the same
and are denoted by $d(r(n)).$ The set of indices $n>N$ for which the periodic
eigenvalues lying in $d(r(n))$ are simple (double) is denoted by
$\mathbb{N}_{1}$ $\,$($\mathbb{N}_{2}$). If $n\in\mathbb{N}_{2}$ then
$\lambda_{n,1}=\lambda_{n,2\text{ }}$ and these eigenvalues are redenoted by
$\lambda_{n}.$

\begin{theorem}
Let $q\in L_{1}[0,1]$ and $N$ be a large number defined in Remark 1.

$(a)$ If $n\in\mathbb{N}_{2}$ and the inequality \textit{ }%
\begin{equation}
\mid q_{2n}+B(\lambda)\mid+\mid q_{-2n}+B^{\prime}(\lambda)\mid\neq0,
\end{equation}
holds for $\lambda=\lambda_{n}$, then the geometric multiplicity of the
eigenvalue $\lambda_{n}$ is $1$.\textit{ }

$(b)$ \textit{Suppose for }$n>N$ there exists an eigenvalue of $L_{0}(q)$
lying in $d(r(n))$ and denoted, for simplicity of notation, by $\lambda_{n,1}%
$\textit{ such that (29) for }$\lambda=\lambda_{n,1}$ \textit{holds. }Then the
normal system of EAF of $L_{0}(q)$ forms a Riesz basis if and only if
\begin{equation}
q_{2n}+B(\lambda_{n,1})\sim\mathit{\ }q_{-2n}+B^{\prime}(\lambda_{n,1}).
\end{equation}

$(c)$ If (29) for $\lambda=\lambda_{n,1}$, where $n>N,$ and (30) hold, then
the large periodic eigenvalues are simple and the system of root functions of
$L_{0}(q)$ forms a Riesz basis.
\end{theorem}

\begin{proof}
$(a)$ Suppose that there exist $2$\ eigenfunctions corresponding to
$\lambda_{n}$. Then one can choose the eigenfunction $\Psi_{n}$ such that
$\left(  \Psi_{n},e^{i2\pi nx}\right)  =0.$ This with (11) \ and (25) implies
that $q_{2n}+B(\lambda_{n})=0$. In the same way we prove that $q_{-2n}%
+B^{\prime}(\lambda_{n})=0.$ The last two equalities contradict (29).

$(b)$ If (29) for $\lambda=\lambda_{n,1}$ and (30) hold, then one can readily
see that
\begin{equation}
\text{ }q_{-2n}+B(\lambda_{n,1})\neq0,\text{ }q_{-2n}+B^{\prime}(\lambda
_{n,1})\neq0.
\end{equation}
These with formulas (11), (15) and (25) imply that
\begin{equation}
u_{n,1}(t)v_{n,1}(t)\neq0
\end{equation}
for $t=0.$ Indeed, if $u_{n,1}(0)=0$ then by (25) $v_{n,1}(0)\neq0$ and by
(11) $q_{2n}+B(\lambda_{n,1})=0$ which contradicts (31). Similarly, if
$v_{n,1}(0)=0$ then by (25) and (15) $q_{-2n}+B^{\prime}(\lambda_{n,1})=0$
which again contradicts (31). By (31) and (32) the right-hand sides of (11)
and (15) for $t=0$ are not zero. Therefore, dividing (11) and (15) side by
side and using the equality $A(\lambda_{n,1})=A^{\prime}(\lambda_{n,1})$ (see
the proof of Lemma 3 of [11]), we get
\begin{equation}
\frac{q_{-2n}+B^{\prime}(\lambda_{n,1})}{q_{2n}+B(\lambda_{n,1})}%
=\frac{u_{n,1}^{2}(0)}{v_{n,1}^{2}(0)}.
\end{equation}
Then, by (30) and (25) we have
\begin{equation}
u_{n,1}(0)\sim v_{n,1}(0)\sim1
\end{equation}
which implies that the set of the Jordan chains is finite (see the end of page
118 of [1]). Thus, if (29) and (30) hold, then the large periodic eigenvalues
are simple. Moreover, by Theorem 1 of [11] the relation (34) implies that the
normal system of EAF of $L_{0}(q)$ form a Riesz basis.

Now suppose that (29) for $\lambda=\lambda_{n,1}$ holds and the normal system
of EAF of $L_{0}(q)$ form a Riesz basis. By Theorem 1 of [11] the set of the
Jordan chains is finite and (34) holds. On the other hand, at least one of the
summands in (29) for $\lambda=\lambda_{n,1}$ are not zero. Suppose, without
less of generality, that $q_{2n}+B(\lambda_{n,1})\neq0.$ Then, using (11),
(15) for $t=0$ and (34) we see that equality (33) holds (see the proof of
(33)). Therefore, using (34) we obtain (30).

$(c)$ The simplicity of the large periodic eigenvalues is proved in $(b).$
Therefore it is enough to note that in this case the Riesz basis property of
the root functions follows from the Riesz basis property of the normal system
of EAF.
\end{proof}

Now, we consider the case $t\in\lbrack0,\rho].$

\begin{theorem}
\textit{ A number }$\lambda\in U(n,t,\rho)$ is an eigenvalue of $L_{t}(q)$ for
$t\in\lbrack0,\rho]$ and $n>N,$ where $U(n,t,\rho)$ and $N$ are defined in
Remark \ 1, if and only if
\begin{equation}
(\lambda-(2\pi n+t)^{2}-A(\lambda,t))(\lambda-(2\pi n-t)^{2}-A^{\prime
}(\lambda,t))=(q_{2n}+B(\lambda,t))(q_{-2n}+B^{\prime}(\lambda,t)).
\end{equation}
Moreover\textit{ }$\lambda\in U(n,t,\rho)$ is a double eigenvalue of $L_{t}$
if and only if it is a double root of (35).
\end{theorem}

\begin{proof}
If $u_{n,j}(t)=0,$ then by (25) we have$\ v_{n,j}(t)\neq0$.\ Therefore, (11)
and (15) imply that $q_{2n}+B(\lambda_{n,j}(t),t)=0$ and $\lambda
_{n,j}(t)-(-2\pi n+t)^{2}-A^{\prime}(\lambda_{n,j}(t),t)=0,$ that is, the
right-hand side and the left-hand side of (35)\ vanish \ when $\lambda$ is
replaced by$\ \lambda_{n,j}(t)$. Hence $\lambda_{n,j}(t)$ satisfies (35). In
the same way we prove that if $\ \ v_{n,j}(t)=0$ then $\lambda_{n,j}(t)$ is a
\ root of (35). It remains to consider the case $u_{n,j}(t)v_{n,j}(t)\neq0.$
In this case multiplying (11) and (15) side by side and canceling
$u_{n,j}(t)v_{n,j}(t)$ we get an equality obtained from (35) by replacing
$\lambda$ with $\lambda_{n,j}(t).$\ Thus, in any case $\lambda_{n,j}(t)$ is a
root of\ (35).

Now we prove that the roots of (35) lying in $U(n,t,\rho)$\ are the
eigenvalues of $L_{t}(q).$ Let $F(\lambda,t,)$ be the left-hand side minus the
right-hand side of (35). Using (20) one can easily verify that the inequality
\
\begin{equation}
\mid F(\lambda,t,)-G(\lambda,t)\mid<\mid G(\lambda,t)\mid,
\end{equation}
where $G(\lambda,t)=(\lambda-(2\pi n+t)^{2})(\lambda-(2\pi n-t)^{2}),$ holds
for all $\lambda$ from the boundary of $U(n,t,\rho).$ Since the function
$(\lambda-(2\pi n+t)^{2})(\lambda-(2\pi n-t)^{2})$ has two roots in \ the
\ set $U(n,t,\rho),$ by the Rouche's theorem from (36) we obtain that
$F(\lambda,t,)$ has two roots in the same\ set.\ Thus\ $L_{t}(q)$ has two
eigenvalue (counting with multiplicities) lying in $U(n,t,\rho)$ (see Remark
1) that are the roots of (35). On the other hand, (35) has preciously two
roots (counting with multiplicities) in $U(n,t,\rho).$ Therefore $\lambda\in
U(n,t,\rho)$ is an eigenvalue of $L_{t}(q)$ if and only if (35) holds.

If \textit{ }$\lambda\in U(n,t,\rho)$ is a double eigenvalue of $L_{t}(q),$
then by Remark 1 $L_{t}(q)$ has no other eigenvalues\textit{ }in\textit{
}$U(n,t,\rho)$ and hence (35) has no other roots. This implies that $\lambda$
is a double root of (35). By the same argument one can prove that if $\lambda$
is a double root of (35) then it is a double eigenvalue of $L_{t}(q)$
\end{proof}

One can readily verify that equation (35) can be written in the form
\begin{equation}
(\lambda-(2\pi n+t)^{2}-\frac{1}{2}(A+A^{\prime})+4\pi nt)^{2}=D,
\end{equation}
where
\begin{equation}
D(\lambda,t)=(4\pi nt)^{2}+q_{2n}q_{-2n}+8\pi ntC+C^{2}+q_{2n}B^{\prime
}+q_{-2n}B+BB^{\prime}%
\end{equation}
and, for brevity, we denote $C(\lambda,t),$ $\ B(\lambda,t),$ $\ A(\lambda,t)$
etc. by $C,$ $\ B,$ $A$ etc. It is clear that $\lambda$ is a root of (37) if
and only if \ it satisfies at least one of the equations
\begin{equation}
\lambda-(2\pi n+t)^{2}-\frac{1}{2}(A(\lambda,t)+A^{\prime}(\lambda,t))+4\pi
nt=-\sqrt{D(\lambda,t)}%
\end{equation}
and
\begin{equation}
\lambda-(2\pi n+t)^{2}-\frac{1}{2}(A(\lambda,t)+A^{\prime}(\lambda,t))+4\pi
nt=\sqrt{D(\lambda,t)},
\end{equation}
where
\begin{equation}
\sqrt{D}=\sqrt{\left\vert D\right\vert }e^{(\arg D)/2},\text{ }-\pi<\arg
D\leq\pi.
\end{equation}

\begin{remark}
It is clear from the construction of $D(\lambda,t)$ that this function is
continuous with respect to $(\lambda,t)$ for $t\in\lbrack0,\rho]$ and
$\lambda\in U(n,t,\rho).$ Moreover, by Remark 1 the eigenvalues $\lambda
_{n,1}(t)$ and $\lambda_{n,2}(t)$ continuously depend on $t\in\lbrack0,\rho].$
Therefore $D(\lambda_{n,j}(t),t)$ for $n>N$ and $j=1,2$ is a continuous
functions of $t\in\lbrack0,\rho].$ By (38), (23), (12), (16) and (19) we have
\[
D(\lambda_{n,j}(\rho),\rho)=(4\pi nt)^{2}+o(1),\text{ }A(\lambda_{n,j}%
(\rho),\rho)+A^{\prime}(\lambda_{n,j}(\rho),\rho)=o(1)
\]
as $n\rightarrow\infty.$ Therefore by (7) and \textit{Theorem 2 of [13]} the
eigenvalues $\lambda_{n,1}(\rho)$ and $\lambda_{n,2}(\rho)$ are simple,
$\lambda_{n,1}(\rho),$ satisfies (39) and $\lambda_{n,2}(\rho)$ satisfies
(40). If $\lambda_{n,1}(t)$ and $\lambda_{n,2}(t)$ are simple for $t\in\lbrack
t_{0},\rho],$ where $0\leq t_{0}\leq\rho,$ then these functions are analytic
function on $[t_{0},\rho]$ and $\lambda_{n,1}(t)\neq\lambda_{n,2}(t)$ for all
$t\in\lbrack t_{0},\rho]$.
\end{remark}

\begin{theorem}
Suppose that $\sqrt{D(\lambda_{n,j}(t),t))}$ continuously depends on $t$ at
$[t_{0},\rho]$ and
\begin{equation}
D(\lambda_{n,j}(t),t)\neq0,\text{ }\forall t\in\lbrack t_{0},\rho]
\end{equation}
for $n>N$ and $j=1,2,$ where $\rho$ and $N$ are defined in Remark \ 1 and
$\sqrt{D}$ is defined in (41) and $0\leq t_{0}\leq\rho$. Then for $t\in\lbrack
t_{0},\rho]$ the eigenvalues $\lambda_{n,1}(t)$ and $\lambda_{n,2}(t)$ defined
in Remark 1 are simple, $\lambda_{n,1}(t)$ satisfies (39) and $\lambda
_{n,2}(t),$ satisfies (40). That is
\begin{equation}
\lambda_{n,j}(t)=(2\pi n+t)^{2}+\frac{1}{2}(A(\lambda_{n,j},t)+A^{\prime
}(\lambda_{n,j},t))-4\pi nt+(-1)^{j}\sqrt{D(\lambda_{n,j},t)}%
\end{equation}
for $t\in\lbrack t_{0},\rho],$ $n>N$ and $j=1,2.$
\end{theorem}

\begin{proof}
By Remark 2, the eigenvalues $\lambda_{n,1}(\rho)$ and $\lambda_{n,2}(\rho)$
are simple, $\lambda_{n,1}(\rho)$ satisfies (39) and $\lambda_{n,2}(\rho)$
satisfies (40). Let us we prove that $\lambda_{n,1}(t)$ satisfies (39) for all
$t\in\lbrack t_{0},\rho]$. Suppose to the contrary that this claim is not
true. Then there exists $t\in\lbrack t_{0},\rho)$ and the sequences
$p_{n}\rightarrow t$ and $q_{n}\rightarrow t,$ where one of them may be a
constant sequence, such that $\lambda_{n,1}(p_{n})$ and $\lambda_{n,1}(q_{n})$
satisfy (39) and (40) respectively. Using the continuity of $\sqrt
{(D(\lambda_{n,j}(t),t))}$, we conclude that $\lambda_{n,1}(t)$ satisfies both
(39) and (40). However, it is possible only if $D(\lambda_{n,1}(t),t)=0$ which
contradicts (42). Hence $\lambda_{n,1}(t)$ satisfies (39) for all $t\in\lbrack
t_{0},\rho]$. In the same way we prove that $\lambda_{n,2}(t)$ satisfies (40)
for all $t\in\lbrack t_{0},\rho]$. If $\lambda_{n,1}(t)=\lambda_{n,2}(t)$ for
some value of $t\in\lbrack t_{0},\rho]$, that is if $\lambda_{n,j}(t)$ is a
double eigenvalue then it satisfies both (39) and (40) which again contradicts (42)
\end{proof}

\section{On the Operator $H_{t}$ for $t\in(-\pi,\pi].$}

In this section we study the operator $H_{t}$ for $t\in\lbrack0,\rho].$ Note
that we consider only the case $t\in\lbrack0,\rho]$ due to the following
reason. The case $t\in\lbrack\rho,\pi-\rho]$ was considered in [13]. \ The
case $t\in\lbrack\pi-\rho,\pi]$ is similar to the case $t\in\lbrack0,\rho]$
and we explain it in Remark 3. Besides, the eigenvalues of $H_{-t}$ coincides
with the eigenvalues of $H_{t}.$

When the potential $q$ has the form (4) then by (6)
\begin{equation}
q_{-1}=a,\text{ }q_{1}=b,\text{ }q_{n}=0,\text{ }\forall n\neq\pm1
\end{equation}
and hence formulas (11), (15), (37) and (38) have the form
\begin{equation}
(\lambda_{n,j}(t)-(2\pi n+t)^{2}-A(\lambda_{n,j}(t),t))u_{n,j}(t)=B(\lambda
_{n,j}(t),t)v_{n,j}(t),
\end{equation}%
\begin{equation}
(\lambda_{n,j}(t)-(-2\pi n+t)^{2}-A^{\prime}(\lambda_{n,j}(t),t))v_{n,j}%
(t)=B^{\prime}(\lambda_{n,j}(t),t)u_{n,j}(t),
\end{equation}%
\begin{equation}
(\lambda-(2\pi n+t)^{2}-\frac{1}{2}(A(\lambda,t)+A^{\prime}(\lambda,t))+4\pi
nt)^{2}=D(\lambda,t),
\end{equation}%
\begin{equation}
D(\lambda,t)=(4\pi nt+C(\lambda,t))^{2}+B(\lambda,t)B^{\prime}(\lambda,t).
\end{equation}
Moreover, by Theorem 2,\textit{ }$\lambda\in U(n,t,\rho)$ is a double
eigenvalue of $H_{t}$ if and only if it satisfies (47) and the equation
\begin{equation}
2(\lambda-(2\pi n+t)^{2}-\frac{1}{2}(A+A^{\prime})+4\pi nt)^{2}(1-\frac{1}%
{2}\frac{\partial}{\partial\lambda}(A+A^{\prime}))=\frac{\partial}%
{\partial\lambda}(D(\lambda,t)).
\end{equation}
By, (28) and (44) $\lambda_{n,j}(t)\in d^{-}(2Kn^{-1},t)\cup d^{+}%
(2Kn^{-1},t)\subset U(n,t,\rho).$ Therefore the formula
\begin{equation}
\lambda_{n,j}(t)=(2\pi n)^{2}+O(n^{-1})
\end{equation}
holds uniformly, with respect to $t\in\lbrack0,n^{-2}],$ for $j=1,2,$ i.e.,
there exist positive constants $M$ and $N$ such that $\mid\lambda
_{n,j}(t)-(2\pi n)^{2}\mid<Mn^{-1}$ for $n\geq N$ \ and $t\in\lbrack
0,n^{-2}].$

Let us consider the functions taking part in (45)-(48). From (44) we see that
the indices in formulas (13), (14) for the case (4) satisfy the conditions
\begin{equation}
\{n_{1},n_{2},...,n_{k}\}\subset\{-1,1\},\text{ }n_{1}+n_{2}+...+n_{s}%
\neq0,2n,
\end{equation}%
\begin{equation}
\{n_{1},n_{2},...,n_{k},2n-n_{1}-n_{2}-...-n_{k}\}\subset\{-1,1\},\text{
}n_{1}+n_{2}+...+n_{s}\neq0,2n
\end{equation}
for $s=1,2,...,k$ respectively. Hence, by (44) $q_{-n_{1}-n_{2}-...-n_{k}}=0$
if $k$ is an even number. Therefore, by (13) and (17)%
\begin{equation}
a_{2m}(\lambda,t)=0,\text{ }a_{2m}^{\prime}(\lambda,t)=0,\text{ }\forall
m=1,2,...
\end{equation}
Since the indices $n_{1},n_{2},...,n_{k}$ take two values (see (51)) the
number of the summands in the right-hand side of (13) is not more than
$2^{k}.$ Clearly, these summands for $k=2m-1$ have the form
\[
a_{k}(\lambda,n_{1},n_{2},...,n_{k},t)=:(ab)^{m}%
{\textstyle\prod\limits_{s=1,2,...,k}}
\left(  \lambda-(2\pi(n-n_{1}-n_{2}-...-n_{s})+t)^{2}\right)  ^{-1}%
\]
(see (13) and (51)). Therefore, we have
\begin{equation}
\text{ }a_{2m-1}(\lambda_{n,j}(t),t)=(4ab)^{m}O(n^{-2m+1}).
\end{equation}
If $t\in\lbrack0,n^{-2}],$ then one can readily see that
\[
a_{1}(\lambda_{n,j},t)=\frac{ab}{(2\pi n)^{2}+O(n^{-1})-(2\pi(n-1))^{2}}%
+\frac{ab}{(2\pi n)^{2}+O(n^{-1})-(2\pi(n+1))^{2}}%
\]%
\[
=\frac{ab}{2\pi(2\pi(2n-1)}-\frac{ab}{2\pi(2\pi(2n+1)}+O(\frac{1}{n^{3}%
})=O(\frac{1}{n^{2}}).
\]
The same estimations for $a_{2m-1}^{\prime}(\lambda_{n,j}(t),t)$ and
$a_{1}^{\prime}(\lambda_{n,j}(t),t)$ hold respectively. Thus, by (12), (16),
(19) and (53), we have
\begin{equation}
A(\lambda_{n,j}(t),t)=O(n^{-2}),\text{ }A^{\prime}(\lambda_{n,j}%
(t),t)=O(n^{-2}),\text{ }\forall t\in\lbrack0,n^{-2}].\text{ }%
\end{equation}

Now we study the functions $B(\lambda,t)$ and $B^{\prime}(\lambda,t)$ (see
(12), (14) and (16), (18)). First let us consider $b_{2n-1}(\lambda,t).$ If
$k=2n-1,$ then by (52) $n_{1}=n_{2}=...=n_{2k-1}=1.$ Using this and (44) in
(14) for $k=2n-1,$ we obtain
\begin{equation}
b_{2n-1}(\lambda,t)=b^{2n}%
{\textstyle\prod\limits_{s=1}^{2n-1}}
\left(  \lambda-(2\pi(n-s)+t)^{2}\right)  ^{-1}.
\end{equation}
If\ $k<2n-1$ or $k=2m,$ then, by (44), $q_{2n-n_{1}-n_{2}-...-n_{k}}=0$ and by
(14)
\begin{equation}
\text{ }b_{k}(\lambda,t)=0.
\end{equation}
In the same way, from (18) we obtain
\begin{equation}
b_{2n-1}^{\prime}(\lambda,t)=a^{2n}%
{\textstyle\prod\limits_{s=1}^{2n-1}}
\left(  \lambda-(2\pi(n-s)-t)^{2}\right)  ^{-1},\text{ \ }b_{k}^{\prime
}(\lambda_{n,j}(t),t)=0
\end{equation}
for\ $k<2n-1$ or $k=2m.$ Now, (19), (57) and (58) imply that the equalities
\begin{equation}
B(\lambda,t)=O\left(  n^{-5}\right)  ,\text{ }B^{\prime}(\lambda,t)=O\left(
n^{-5}\right)
\end{equation}
hold uniformly for $t\in\lbrack0,\rho]$ and $\lambda\in U(n,t,\rho).$ From
(45) and (46) (if $\left\vert u_{n,j}(t)\right\vert \geq\left\vert
v_{n,j}(t)\right\vert $ then use (45) and if $\left\vert v_{n,j}(t)\right\vert
>\left\vert u_{n,j}(t)\right\vert $ then use (46)) by using (55) and (59) we
obtain that the formula
\begin{equation}
\lambda_{n,j}(t)=(2\pi n)^{2}+O(n^{-2})
\end{equation}
holds uniformly, with respect to $t\in\lbrack0,n^{-3}],$ for $j=1,2.$

More detail estimations of $B$ and $B^{\prime}$ are given in the following lemma.

\begin{lemma}
If $q$\ has the form (4), then the formulas
\begin{equation}
B(\lambda,t)=\beta_{n}\left(  1+O(n^{-2})\right)  ,\text{ }B^{\prime}%
(\lambda,t)=\alpha_{n}\left(  1+O(n^{-2})\right)  ,
\end{equation}%
\begin{equation}
\frac{\partial}{\partial\lambda}(B^{\prime}(\lambda,t)B(\lambda,t))\sim
\alpha_{n}\beta_{n}n^{-1}\ln\left\vert n\right\vert
\end{equation}
hold uniformly for
\begin{equation}
t\in\lbrack0,n^{-3}],\text{ }\lambda=(2\pi n)^{2}+O(n^{-2}),
\end{equation}
where $\beta_{n}=b^{2n}\left(  (2\pi)^{2n-1}(2n-1)!\right)  ^{-2}$ and
$\alpha_{n}=a^{2n}\left(  (2\pi)^{2n-1}(2n-1)!\right)  ^{-2}.$
\end{lemma}

\begin{proof}
Using (56) and (58) by direct calculations we get
\begin{equation}
b_{2n-1}((2\pi n)^{2},0)=\beta_{n},\text{ }b_{2n-1}^{\prime}((2\pi
n)^{2},0)=\alpha_{n}.
\end{equation}
If $1\leq s\leq2n-1$ then for any $(\lambda,t)$ satisfying (63) there exists
$\lambda_{1}=(2\pi n)^{2}+O(n^{-2})$ and

$\lambda_{2}=(2\pi n)^{2}+O(n^{-2})$ such that
\begin{equation}
\mid\lambda_{1}-(2\pi(n-s))^{2}\mid<\mid\lambda-(2\pi(n-s)+t)^{2}\mid
<\mid\lambda_{2}-(2\pi(n-s))^{2}\mid.
\end{equation}
Therefore from (56) we obtain that
\begin{equation}
\left\vert b_{2n-1}(\lambda_{1},0)\right\vert <\left\vert b_{2n-1}%
(\lambda,t)\right\vert <\left\vert b_{2n-1}(\lambda_{2},0)\right\vert .
\end{equation}
On the other hand, differentiating (56) with respect to $\lambda,$ we conclude
that
\begin{equation}
\frac{\partial}{\partial\lambda}(b_{2n-1}((2\pi n)^{2},0))=b_{2n-1}((2\pi
n)^{2},0)\sum_{s=1}^{2n-1}\frac{1+O(n^{-1})}{s(2n-s)}.
\end{equation}
Now taking into account that the last summation is of order $n^{-1}%
\ln\left\vert n\right\vert $ and using (64), we get
\begin{equation}
\frac{\partial}{\partial\lambda}b_{2n-1}((2\pi n)^{2},0))\sim\beta_{n}%
n^{-1}\ln\left\vert n\right\vert .\text{ }%
\end{equation}
Arguing as above one can easily see that the $m$-th derivative, where
$m=2,3,...,$ of $b_{2n-1}(\lambda,0)$ is $O(\beta_{n}).$ Hence using the
Taylor series of $b_{2n-1}(\lambda,0)$ for $\lambda=(2\pi n)^{2}+O(n^{-2})$
about $(2\pi n)^{2},$ we obtain $b_{2n-1}(\lambda_{i},0)=\beta_{n}%
(1+O(n^{-2})),$ $\forall i=1,2.$ This with (66) yields%
\begin{equation}
b_{2n-1}(\lambda,t)=\beta_{n}(1+O(n^{-2}))
\end{equation}
for all $(\lambda,t)$ satisfying (63). In the same way, we get
\begin{equation}
\frac{\partial}{\partial\lambda}b_{2n-1}^{\prime}((2\pi n)^{2},0))\sim
\alpha_{n}(\frac{\ln n}{n}),\text{ }b_{2n-1}^{\prime}(\lambda,t)=\alpha
_{n}(1+O(n^{-2})).\text{ }%
\end{equation}

Now let us consider $b_{2n+1}(\lambda,t).$ By (52) the indices $n_{1}%
,n_{2},...,n_{2n+1}$ taking part in $b_{2n+1}(\lambda,t)$ are $\ 1$ except
one, say $n_{s+1}=-1,$ where $s=2,3,...,2n-1.$ Moreover, if $n_{s+1}=-1,$ then
$n_{1}+n_{2}+...+n_{s+1}=n_{1}+n_{2}+...+n_{s-1}=s-1$ and

$n_{1}+n_{2}+...+n_{s+2}=n_{1}+n_{2}+...+n_{s}=s.$ Therefore, by (14),
$b_{2n+1}(\lambda,t)$ for
\begin{equation}
\lambda=(2\pi n)^{2}+O(n^{-1}),\text{ }t\in\lbrack0,n^{-3}]
\end{equation}
has the form
\[
b_{2n-1}(\lambda,t)%
{\textstyle\sum\limits_{s=2}^{2n-1}}
\frac{ab}{(2\pi n)^{2}-(2\pi(n-s+1))^{2}+O(n^{-1}))(2\pi n)^{2}-(2\pi
(n-s))^{2}+O(n^{-1}))}.
\]
One can easily see that the last sum is $O(n^{-2}).$ Thus we have
\begin{equation}
b_{2n+1}(\lambda,t)=b_{2n-1}(\lambda,t)O(n^{-2})=\beta_{n}O(n^{-2}))
\end{equation}
for all $(\lambda,t)$ satisfying (71).

Now let us estimate $b_{k}(\lambda,t)$ for $k>2n+1$. Since the sums in (14)
are taken under conditions (52), we conclude that $1\leq n_{1}+n_{2}%
+\cdots+n_{s}\leq2n-1.$ Using this instead of $1\leq s\leq2n-1$ and repeating
the proof of (66) we obtain that for any $(\lambda,t)$ satisfying (71) there
exists $\lambda_{3}=(2\pi n)^{2}+O(n^{-1})$ and $\lambda_{4}=(2\pi
n)^{2}+O(n^{-1})$ such that%
\[
\left\vert p_{n_{1},n_{2},...,n_{k}}(\lambda_{3},0)\right\vert <\left\vert
p_{n_{1},n_{2},...,n_{k}}(\lambda,t)\right\vert <\left\vert p_{n_{1}%
,n_{2},...,n_{k}}(\lambda_{4},0)\right\vert ,\text{ }\forall k<2n-1,
\]
where $p_{n_{1},n_{2},...,n_{k}}(\lambda,0)$ is defined in (26). This with
(27) and (72) implies that
\begin{equation}
\sum_{k=2n+1}^{\infty}\left\vert b_{k}(\lambda,t)\right\vert =\beta
_{n}O(n^{-2})
\end{equation}
for all $(\lambda,t)$ satisfying (71). In the same way, we obtain
\begin{equation}
\sum_{k=2n+1}^{\infty}\left\vert b_{k}^{\prime}(\lambda,t)\right\vert
=\alpha_{n}O(n^{-2}).
\end{equation}
Thus (61) follows from (69), (70), (73) and (74).

Now we prove (62). It follows from (73), (74) and the Cauchy inequality that
\begin{equation}
\frac{\partial}{\partial\lambda}(\sum_{k=2n+1}^{\infty}b_{k}(\lambda
,t))=\beta_{n}O(n^{-1}),\text{ }\frac{\partial}{\partial\lambda}(\sum
_{k=2n+1}^{\infty}b_{k}^{\prime}(\lambda,t))=\alpha_{n}O(n^{-1}).
\end{equation}
Therefore (62) follows from (68) and (70).
\end{proof}

In the case (4) formula (61) with (44) implies that the inequality (29) for
$\lambda=\lambda_{n,j}(0)$ holds and the relation (30) holds if and only if
$\mid a\mid=\mid b\mid.$ On the other hand, the large eigenvalues of $H_{0}$
and $H_{\pi}$ are simple (see, for example, theorems 11 and 12 of [16] and
Theorem 1 of [15]). Therefore, taking into account that, if the number of
multiple eigenvalues is finite then Riesz basis property of the root functions
follows from the Riesz basis property of the normal system of EAF we get the
following consequence of Theorem 1.

\begin{corollary}
The root functions of $H_{0}$ form a Riesz basis if and only if $\mid
a\mid=\mid b\mid.$ \ The statement continue to hold if $H_{0}$ is replaced by
$H_{\pi}$ which can be proved in the same way.
\end{corollary}

These result were obtained by Djakov and Mitjagin [3] by the other method.

From Lemma 1 it easily follows also the following statement.

\begin{proposition}
If $\lambda_{n,j}(t)$ for $t\in\lbrack0,\rho]$ is a multiple eigenvalue of
$H_{t}$ then
\begin{equation}
(4\pi nt)^{2}=-\beta_{n}\alpha_{n}\left(  1+O(n^{-2})\right)  .
\end{equation}

\end{proposition}

\begin{proof}
If $\lambda_{n,1}(t)=\lambda_{n,2}(t)=:\lambda_{n}(t)$ is a multiple
eigenvalue, then as it is noted in the beginning of this section, it satisfies
(47) and (49) from which we obtain%
\begin{equation}
4D(\lambda_{n}(t),t)\left(  1-\frac{1}{2}\frac{\partial}{\partial\lambda
}(A(\lambda_{n}(t),t)+A^{\prime}(\lambda_{n}(t),t))\right)  ^{2}=\left(
\frac{\partial}{\partial\lambda}D(\lambda_{n}(t),t)\right)  ^{2}.
\end{equation}
By (21) and (23) we have
\begin{equation}
\frac{\partial}{\partial\lambda}(A(\lambda_{n}(t),t)+A^{\prime}(\lambda
_{n}(t),t))=O(n^{-2}),
\end{equation}%
\begin{align}
(4\pi nt+C(\lambda_{n}(t),t))^{2}  &  =(4\pi nt)^{2}(1+O(n^{-2})),\text{ }\\
\frac{\partial}{\partial\lambda}(4\pi nt+C(\lambda_{n}(t),t))^{2}  &  =(4\pi
nt)^{2}(1+O(n^{-2}))O(n^{-3})
\end{align}
for $t\in\lbrack0,\rho]$. On the other hand it follows from (59) and (22) that%
\begin{equation}
B(\lambda_{n}(t),t)B^{\prime}(\lambda_{n}(t),t)=O(n^{-10}),\text{ }%
\frac{\partial}{\partial\lambda}(B^{\prime}(\lambda_{n}(t),t)B(\lambda
_{n}(t),t))=O(n^{-7}).
\end{equation}
Therefore from (48) and (79)-(81) we obtain
\begin{equation}
D((\lambda_{n}(t),t))=(4\pi nt)^{2}(1+O(n^{-2}))+O(n^{-10})
\end{equation}
and
\begin{equation}
\frac{\partial}{\partial\lambda}(D(\lambda_{n}(t),t))=(4\pi nt)^{2}%
(1+O(n^{-2}))O(n^{-3})+O(n^{-7}).
\end{equation}
Using the equalities (78), (82) and (83) in (77) we get
\begin{equation}
4(4\pi nt)^{2}(1+O(n^{-2}))=(4\pi nt)^{2}O(n^{-4})+O(n^{-8}).
\end{equation}
Hence, we have $t\in\lbrack0,n^{-3}].$ Then by (60), $t$ and $\lambda
=:\lambda_{n}(t)$ satisfy (63) and by Lemma 1
\begin{equation}
B(\lambda_{n}(t),t)=\beta_{n}\left(  1+O(n^{-2})\right)  ,\text{ }B^{\prime
}(\lambda_{n}(t),t)=\alpha_{n}\left(  1+O(n^{-2})\right)  ,
\end{equation}%
\begin{equation}
\frac{\partial}{\partial\lambda}(B^{\prime}(\lambda_{n}(t),t)B(\lambda
_{n}(t),t))\sim\alpha_{n}\beta_{n}n^{-1}\ln\left\vert n\right\vert .
\end{equation}
Therefore by (48), (79) and (80) we have
\begin{equation}
D((\lambda_{n}(t),t))=(4\pi nt)^{2}(1+O(n^{-2}))+\beta_{n}\alpha_{n}\left(
1+O(n^{-2})\right)
\end{equation}
and
\begin{equation}
\frac{\partial}{\partial\lambda}(D(\lambda_{n}(t),t))=(4\pi nt)^{2}%
(1+O(n^{-2}))O(n^{-3})+O(\alpha_{n}\beta_{n}n^{-1}\ln\left\vert n\right\vert
).
\end{equation}
Now using (78), (87) and (88) in (77) we obtain
\[
(4\pi nt)^{2}(1+O(n^{-2}))+\beta_{n}\alpha_{n}\left(  1+O(n^{-2})\right)
=(4\pi nt)^{2}O(n^{-4})+\left(  O(\alpha_{n}\beta_{n}n^{-1}\ln\left\vert
n\right\vert \right)  )^{2}%
\]
which implies (76)
\end{proof}

Now we are ready to prove the main result of this section by using Theorem 3.
In formula (87) the terms $O(n^{-2})$ do not depend on $t$, that is, there
exists $c>0$ such that these terms satisfy the inequality
\begin{equation}
\left\vert O(n^{-2})\right\vert <cn^{-2}.
\end{equation}

\begin{theorem}
Let $\mathbb{S}$ be the set of integer $n>N$ such that
\begin{equation}
-\pi+3cn^{-2}\leq\arg(\beta_{n}\alpha_{n})\leq\pi-3cn^{-2}%
\end{equation}
and $\left\{  t_{n}:n>N\right\}  $ be a sequence defined as follows: $t_{n}=0$
if $n\in$ $\mathbb{S}$ and
\begin{equation}
(4\pi nt_{n})^{2}(1-cn^{-2})=-(1+cn^{-2}+n^{-3})\operatorname{Re}(\beta
_{n}\alpha_{n})
\end{equation}
if $n\notin$ $\mathbb{S},$ where $c$ is defined in (89). Then the eigenvalues
$\lambda_{n,1}(t)$ and $\lambda_{n,2}(t)$ defined in Remark 1 are simple and
satisfy (43) for $t\in\lbrack t_{n},\rho].$
\end{theorem}

\begin{proof}
Let $n\notin$ $\mathbb{S}.$ It follows from (48), (59) and (79) that if $t\geq
n^{-3}$ then
\begin{equation}
\operatorname{Re}D(\lambda_{n,j}(t),t))>0.
\end{equation}
If $t\in\lbrack0,n^{-3}]$ then we have formula (87). Since the terms
$O(n^{-2})$ in (87) satisfy (89) we have the following estimate for the real
part of the first term in the right-hand side of (87):
\begin{equation}
\operatorname{Re}((4\pi nt)^{2}(1+O(n^{-2})))>(4\pi nt)^{2}(1-cn^{-2}%
)\geq(4\pi nt_{n})^{2}(1-cn^{-2})
\end{equation}
for $t\in\lbrack t_{n},n^{-3}].$ On the other hand if $n\notin$ $\mathbb{S}$
then by the definition of $\mathbb{S}$ (90) does not hold, which implies that
$\beta_{n}\alpha_{n}=-\left\vert (\beta_{n}\alpha_{n})\right\vert e^{i\theta
},$ $\left\vert \theta\right\vert <3cn^{-2}$ and hence
\begin{equation}
\operatorname{Im}(\beta_{n}\alpha_{n})=O(n^{-2})\operatorname{Re}(\beta
_{n}\alpha_{n}).
\end{equation}
Using this and (89), we obtain the following estimate for the real part of the
second term in the right-hand side of (87):%
\[
\left\vert \operatorname{Re}(\beta_{n}\alpha_{n}\left(  1+O(n^{-2})\right)
)\right\vert <(1+cn^{-2}+n^{-3})\left\vert \operatorname{Re}(\beta_{n}%
\alpha_{n})\right\vert .
\]
Therefore it follows from (93), (91) and (87) that (92) holds for $t\in\lbrack
t_{n},n^{-3}]$ , $n>N$ and $n\notin$ $\mathbb{S}.$ Thus (92) is true for all
$t\in\lbrack t_{n},\rho]$. Hence $\sqrt{D(\lambda_{n,j}(t),t))}$ is
well-defined and by Remark 2 it continuously depends on $t$. Therefore the
proof follows from Theorem 3.

Now consider the case $n\in$ $\mathbb{S}.$ By (89) we have
\[
-cn^{-2}-n^{-3}<\arg(1+O(n^{-2}))<cn^{-2}+n^{-3}%
\]
Using (94) and (89) we obtain
\[
-\pi+2cn^{-2}-n^{-3}<\arg(\beta_{n}\alpha_{n}\left(  1+O(n^{-2})\right)
)<\pi-2cn^{-2}+n^{-3},
\]%
\[
-cn^{-2}-n^{-3}<\arg((4\pi nt)^{2}(1+O(n^{-2})))<cn^{-2}+n^{-3}%
\]
and the acute angle between the vectors $(4\pi nt)^{2}((1+O(n^{-2}))$ and
$\beta_{n}\alpha_{n}\left(  1+O(n^{-2})\right)  $ is less than $\pi.$
Therefore by the parallelogram law of vector addition we have
\[
-\pi<\arg(D(\lambda_{n,j}(t),t))<\pi,\text{ }D(\lambda_{n,j}(t),t))\neq0
\]
for $t\in\lbrack0,\rho].$ Thus the proof again follows from Theorem 3
\end{proof}

\begin{corollary}
Suppose that
\begin{equation}
\text{ }\inf_{q,p\in\mathbb{N}}\{\mid2q\alpha-(2p-1)\mid\}\neq0,
\end{equation}
where $\alpha=\pi^{-1}\arg(ab).$ Then for all $n>N$ the eigenvalues
$\lambda_{n,1}(t)$ and $\lambda_{n,2}(t)$ defined in Remark 1 are simple and
satisfy (43) for $t\in\lbrack0,\rho]$.
\end{corollary}

\begin{proof}
By (95) that there exists $\varepsilon>0$ such that $-\pi+\varepsilon
<\arg((ab)^{2n})<\pi-\varepsilon$ for all $n\in\mathbb{N}.$ Hence by the
definition of $\beta_{n}$ and $\alpha_{n}$ (see Lemma 1)
\begin{equation}
-\pi+\varepsilon<\arg(\alpha_{n}\beta_{n})<\pi-\varepsilon,
\end{equation}
that is, (90) holds for all $n>N.$ Therefore the proof follows from Theorem 4
\end{proof}

\begin{remark}
Let $\widetilde{A},$ $\widetilde{B},$ $\widetilde{A}^{\prime}$, $\widetilde
{B}^{\prime}$ and $\widetilde{C}$ be the functions obtained from $A,$ $B,$
$A^{\prime},$ $B^{\prime}$ and $C$ by replacing $a_{k},a_{k}^{\prime}%
,b_{k},b_{k}^{\prime}$ $\ $with $\widetilde{a}_{k},\widetilde{a}_{k}^{\prime
},\widetilde{b}_{k},\widetilde{b}_{k}^{\prime}$ , where $\widetilde{a}%
_{k},\widetilde{a}_{k}^{\prime},\widetilde{b}_{k},\widetilde{b}_{k}^{\prime}$
differ from $a_{k},a_{k}^{\prime},b_{k},b_{k}^{\prime}$ respectively, in the
following sense. The sums in the expressions for $\widetilde{a}_{k}%
,\widetilde{a}_{k}^{\prime},\widetilde{b}_{k},\widetilde{b}_{k}^{\prime}$ are
taken under condition $n_{1}+n_{2}+...+n_{s}\neq0,\pm(2n+1)$ instead of the
condition $n_{1}+n_{2}+...+n_{s}\neq0,\pm2n$ for $s=1,2,...,k.$ In
$\widetilde{b}_{k},\widetilde{b}_{k}^{\prime}$ the multiplicand $q_{\pm
2n-n_{1}-n_{2}-...-n_{k}}$ of $b_{k},b_{k}^{\prime}$ is replaced by
$q_{\pm(2n+1)-n_{1}-n_{2}-...-n_{k}}$. \ To consider the case $t\in\lbrack
\pi-\rho,\pi]$ instead of (11), (15) we use%
\begin{equation}
(\lambda_{n,j}(t)-(2\pi n+t)^{2}-\widetilde{A}(\lambda_{n,j}(t),t))u_{n,j}%
(t)=(q_{2n+1}+\widetilde{B}(\lambda_{n,j}(t),t))v_{n,j}(t),
\end{equation}%
\begin{equation}
(\lambda_{n,j}(t)-(-2\pi(n+1)+t)^{2}-\widetilde{A}^{\prime}(\lambda
_{n,j}(t),t))v_{n,j}(t)=(q_{-2n-1}+\widetilde{B}^{\prime}(\lambda
_{n,j}(t),t))u_{n,j}(t)
\end{equation}
and repeat the investigation of the case $t\in\lbrack0,\rho]$. Note that
instead of (9) for $k\neq0,2n$ using the same inequality for $k\neq0,2n+1$ and
$t\in\lbrack\pi-\rho,\pi]$ from (10) we obtain (97) and (98) instead of (11)
and (15). \ In the case $t\in\lbrack\pi-\rho,\pi]$ \ instead of (43) we
obtain
\begin{equation}
\lambda_{n,j}(t)=(2\pi n+t)^{2}-2\pi(2n+1)(t-\pi)+\frac{1}{2}(\widetilde
{A}^{^{\prime}}+\widetilde{A})+(-1)^{j}\sqrt{\widetilde{D}(\lambda
_{n,j}(t),t))},
\end{equation}
where $\widetilde{D}=\left(  2\pi(2n+1)(t-\pi)+\widetilde{C}\right)
^{2}+\widetilde{B}$ $\widetilde{B}^{\prime}.$ Similarly, instead of (61),
(76), (91) and (95) we obtain respectively the following relations
\[
\widetilde{B}(\lambda,t)=\widetilde{\beta}_{n}\left(  1+O(n^{-2})\right)
,\text{ }\widetilde{B}^{\prime}(\lambda,t)=\widetilde{\alpha}_{n}\left(
1+O(n^{-2})\right)  ,
\]%
\[
\left(  2\pi(2n+1)(t-\pi)\right)  ^{2}=-\widetilde{\beta}_{n}\widetilde
{\alpha}_{n}\left(  1+O(n^{-2})\right)  ,
\]%
\[
(2\pi(2n+1)(\widetilde{t}_{n}-\pi))^{2}(1-cn^{-2})=-(1+cn^{-2}+n^{-3}%
)\operatorname{Re}(\widetilde{\beta}_{n}\widetilde{\alpha}_{n}),
\]%
\[
\text{ }\inf_{q,p\in\mathbb{N}}\{\mid(2q+1)\alpha-(2p-1)\mid\}\neq0,
\]
where $\widetilde{\beta}_{n}=b^{2n+1}\left(  (2\pi)^{2n}(2n)!\right)  ^{-2}$,
$\widetilde{\alpha}_{n}=a^{2n+1}\left(  (2\pi)^{2n}(2n)!\right)  ^{-2},$
$\widetilde{t}_{n}\in\lbrack\pi-\rho,\pi]$ and Proposition 1, Theorem 4,
Corollary 2 continue to hold under the corresponding replacement.

As we noted in Section 2 (see \textit{Theorem 2 of [13] and Remark 1}) the
large eigenvalues of $H_{t}$ for $t\in\lbrack\rho,\pi-\rho]$ consist of the
simple eigenvalues $\lambda_{n}(t)$ for $\left\vert n\right\vert >N$
satisfying the, \textit{uniform with respect to }$t$\textit{ in }$[\rho
,\pi-\rho],$\textit{ asymptotic formula (5).} Thus by Theorem 4 and by the
just noted similar investigation, the eigenvalues $\lambda_{n,j}(t)$ for
$n>N,$ $j=1,2$ and $t\in$ $[t_{n},\rho]\cup$ $[\pi-\rho,\widetilde{t}_{n}]$
and the eigenvalues $\lambda_{n}(t)$ \textit{for \ }$t\in\lbrack\rho,\pi
-\rho]$\textit{ and }$\left\vert n\right\vert >N$\textit{ are simple.}. These
eigenvalues satisfy (43), (5) and (99) for $t\in$ $[t_{n},\rho],$ $t\in
\lbrack\rho,\pi-\rho]$ and $t\in\lbrack\pi-\rho,\widetilde{t}_{n}]$ respectively.
\end{remark}

\section{Asymptotic Analysis of H}

In this section we investigate the operator $H$ generated in $L_{2}%
(-\infty,\infty)$ by (1) when the potential $q$ has the form (4). Since the
spectrum $S(H)$ of $H$ is the union of the spectra $S(H_{t})$ of the operators
$H_{t}$ for $t\in(-\pi,\pi],$ due to the notations of Remark 1, $\Gamma_{n}$
for $n>N$ are the part of $S(H)$ lying in the neighborhood of infinity. Here
we consider only this part of the spectrum.

Following \ [7, 12], we define the projections and the spectral singularities
of $H$ as follows. A closed arc $\gamma=:\{z\in\mathbb{C}:z=\lambda
(t),t\in\lbrack\alpha,\beta]\}$ with $\lambda(t)$ continuous on the closed
interval $[\alpha,\beta]$, analytic in an open neighborhood of $[\alpha
,\beta]$ and $F(\lambda(t))=2\cos t,$
\[
\text{ }\frac{\partial F(\lambda(t))}{\partial\lambda}\neq0,\text{ }\forall
t\in\lbrack\alpha,\beta],\text{ }\lambda^{^{\prime}}(t)\neq0,\text{ }\forall
t\in(\alpha,\beta)
\]
is called a regular spectral arc of $H,$ where $F(\lambda)$ is defined in (3).
The projection $P(\gamma)$ corresponding to the regular spectral arc $\gamma$
is defined by
\begin{equation}
P(\gamma)f=\frac{1}{2\pi}%
{\textstyle\int\limits_{\gamma}}
(\Phi_{+}(x,\lambda)F_{-}(\lambda,f)+\Phi_{-}(x,\lambda)F_{+}(\lambda
,f))\frac{\varphi(1,\lambda)}{p(\lambda)}d\lambda,
\end{equation}
where $\Phi_{\pm}(x,\lambda)=:\theta(x,\lambda)+(\varphi(1,\lambda
))^{-1}(e^{\pm it}-\theta(1,\lambda))\varphi(x,\lambda)$ is the Floquet
solution,
\[
F_{\pm}(\lambda,f)=\int_{\mathbb{R}}f(x)\Phi_{\pm}(x,\lambda)dx,
\]
$p(\lambda)=\sqrt{4-F^{2}(\lambda)}$ and the functions $\theta(x,\lambda)$ and
$\varphi(x,\lambda)$ are defined in (3). Moreover, for the norm of the
projections $P(\gamma)$ corresponding to the regular spectral arc $\gamma$
defined in (100) we use the formula
\begin{equation}
\parallel P(\gamma)\parallel=\sup_{t\in\lbrack\alpha,\beta]}\mid
d(\lambda(t))\mid^{-1}%
\end{equation}
of [12] (see Theorem 2 of [12] or Proposition 1 of [17]), where $d(\lambda
(t))=(\Psi_{\lambda(t)},\Psi_{\lambda(t)}^{\ast}),$ $\Psi_{\lambda(t)}$ and
$\Psi_{\lambda(t)}^{\ast}$ are the normalized eigenfunctions of $H_{t}$ and
$H_{t}^{\ast}$ corresponding to $\lambda(t)$ and $\overline{\lambda(t)}$
respectively. In [14] and [17] the following definitions were given

\begin{definition}
We say that $\lambda\in S(H)$ is a spectral singularity of $H$ if for all
$\varepsilon>0$\ there exists a sequence $\{\gamma_{n}\}$ of the regular
spectral arcs $\gamma_{n}\subset\{z\in\mathbb{C}:\mid z-\lambda\mid
<\varepsilon\}$ such that
\begin{equation}
\lim_{n\rightarrow\infty}\parallel P(\gamma_{n})\parallel=\infty.
\end{equation}

\end{definition}

\begin{definition}
We say that the operator $H$ has a spectral singularity at infinity if there
exists a sequence $\{\gamma_{n}\}$ of the regular spectral arcs such that
$d(0,\gamma_{n})\rightarrow\infty$ as $n\rightarrow\infty$ and (102) holds,
where $d(0,\gamma_{n})$ is the distance from the point $(0,0)$ to the arc
$\gamma_{n}.$
\end{definition}

\begin{definition}
The operator $H$ is said to be an asymptotically spectral operator if there
exists a positive constant $C$ such that
\[
\sup_{\sigma\in R(C)}(ess\sup_{t\in(-\pi,\pi]}\parallel e(t,\sigma
)\parallel)<\infty,
\]
where $R(C)$ is the ring consisting of all sets which are the finite union of
the half closed rectangles lying in $\{\lambda\in\mathbb{C}:\mid\lambda
\mid>C\}$ and $e(t,\sigma)$\textit{ is the spectral projection defined by
contour integration of the resolvent of }$H_{t}$ over $\sigma.$
\end{definition}

\begin{remark}
Since the large eigenvalues of $H_{0}$ and $H_{\pi}$ are simple, Theorem 1 of
[17] for the operator $H$ can be written as follows: The following statements
are equivalent

$(a)$ The operator $H$ has no spectral singularity at infinity.

$(b)$ $H$ is an asymptotically spectral operator.

$(c)$ There exist constant $c_{1}$ and $N$ such that for all $\mid n\mid>N$
and $t\in(-\pi,\pi]$ the eigenvalues $\lambda_{n}(t)$ are simple and
\begin{equation}
\mid d_{n}(t)\mid^{-1}<c_{1}%
\end{equation}
where $d_{n}(t)=(\Psi_{n,t},\Psi_{n,t}^{\ast}),$ $\Psi_{n,t}$ and $\Psi
_{n,t}^{\ast}$ are the normalized eigenfunctions of $H_{t}$ and $H_{t}^{\ast}$
corresponding to the eigenvalues $\lambda_{n}(t)$ and $\overline{\lambda
_{n}(t)}$ respectively.

Note that if $\lambda_{n}(t)$ is a simple eigenvalue then the normalized
eigenfunctions $\Psi_{n,t}$ and $\Psi_{n,t}^{\ast}$ are determined uniquely up
to constant of modulus $1.$ Therefore $\mid d_{n}(t)\mid$ is uniquely defined
and is the norm of the projection defined by integration of the resolvent of
the operator $H_{t}$ over a closed contour containing only the simple
eigenvalue $\lambda_{n}(t)$ . Moreover $\mid d_{n}\mid$ is continuous at $t$
if $\lambda_{n}(t)$ is a simple eigenvalue.
\end{remark}

The following proposition follows from Remark 4, Definition 2 and (101).

\begin{proposition}
The operator $H$ has a spectral singularities at infinity if and only if there
exists a sequence of pairs $\{(n_{k},t_{k})\}$ such that $\lambda_{n_{k}%
}(t_{k})$ is a simple eigenvalue and
\begin{equation}
\lim_{k\rightarrow\infty}d_{n_{k}}(t_{k})=0,
\end{equation}
where $n_{k}\in\mathbb{Z}$ and $t_{k}\in(-\pi,\pi]$.
\end{proposition}

As it was noted in [3] (see page 539 of [3]) if $\mid a\mid\neq\mid b\mid,$
then it follows from [3] and [7] that $H$ is not a spectral operator. This
fact and the following more general fact easily follows from the formulas of
Section 4.

\begin{proposition}
If $\mid a\mid\neq\mid b\mid,$ then the operator $H$ has the spectral
singularity at infinity and hence is not an asymptotically spectral operator.
\end{proposition}

\begin{proof}
Suppose, without loss of generality, that $\mid a\mid<\mid b\mid.$ As was
noted in Remark 4 the large periodic eigenvalues $\lambda_{n}(0)$ are simple.
Due to (44) the formulas (11), (15) and (25) for $t=0$ have the forms
\begin{align*}
(\lambda_{n}(0)-(2\pi n)^{2}-A(\lambda_{n}(0),0))u_{n}  &  =B(\lambda
_{n}(0),0))v_{n},\\
(\lambda_{n}(0)-(2\pi n)^{2}-A^{\prime}(\lambda_{n}(0),0))u_{n}  &
=B^{\prime}(\lambda_{n}(0),0))v_{n},\text{ }\left\vert u_{n}\right\vert
^{2}+\left\vert v_{n}\right\vert ^{2}=1+O(n^{-2})
\end{align*}
where $u_{n}=(\Psi_{n,0},e^{i2\pi nx}),$ $v_{n}=(\Psi_{n,0},e^{-i2\pi nx}).$
Moreover, by (61), $B(\lambda_{n}(0),0))$ and $B^{\prime}(\lambda_{n}(0),0))$
are nonzero numbers and
\[
\frac{B(\lambda_{n}(0),0))}{B^{\prime}(\lambda_{n}(0),0))}=O(\left(  \mid
a\mid/\mid b\mid\right)  ^{n})=O(n^{-2}).
\]
Using these equalities and arguing as in the proof of (33) we obtain
\[
u_{n}=v_{n}O(n^{-1}),\text{ }\Psi_{n,0}(x)=ce^{-i2\pi nx}+O(n^{-1}),
\]
where $\mid c\mid=1$ and $\Psi_{n,0}(x)$ is the normalized eigenfunction
corresponding to $\lambda_{n}(0).$ Replacing $a$ and $b$ by $\overline{b}$ and
$\overline{a}$ respectively, in the same way we obtain
\[
\Psi_{n,0}^{\ast}(x)=\overline{c}e^{i2\pi nx}+O(n^{-1}).
\]
Thus $(\Psi_{n,0},\Psi_{n,0}^{\ast}(x))\rightarrow0$ as $n\rightarrow\infty$
and hence the proof follows from Proposition 2.
\end{proof}

Thus if $\mid a\mid\neq\mid b\mid,$ then $H$ is not a spectral operator. The
inverse statement is not true. If $\mid a\mid=\mid b\mid$ then in the Theorem
6 we will find the necessary and sufficient conditions on $ab$ for the $H$ to
be the asymptotically spectral operator. To prove this main result of this
paper we first prove some preliminary propositions.

\begin{theorem}
Suppose$\mid a\mid=\mid b\mid$ and
\begin{equation}
\text{ }\inf_{q,p\in\mathbb{N}}\{\mid q\alpha-(2p-1)\mid\}\neq0,
\end{equation}
where $\alpha=\pi^{-1}\arg(ab)$. Then there exists $N$ such that for
$\left\vert n\right\vert >N$ the component $\Gamma_{n}$ of the spectrum $S(H)$
of the operator $H$ is a separated simple analytic arc with the end points
$\lambda_{n}(0)$ and $\lambda_{n}(\pi).$ These components do not contain
spectral singularities. In other words, the number of the spectral
singularities of $H$ is finite.
\end{theorem}

\begin{proof}
Corollary 2, Theorem 2 of [13]\textit{,} and Remark 3 immediately imply that
the eigenvalues $\lambda_{n}(t)$ for $\left\vert n\right\vert >N$ and
$t\in(-\pi,\pi]$ are simple. Therefore for $\left\vert n\right\vert >N$ the
component $\Gamma_{n}$ of the spectrum of the operator $H$ is a separated
simple analytic arc with the end points $\lambda_{n}(0)$ and $\lambda_{n}%
(\pi)$. It is well-known that the spectral singularities of $H$ are contained
in the set of multiple eigenvalues of $H_{t}$ (see [7,12]). Hence, $\Gamma
_{n}$ for $\left\vert n\right\vert >N$ does not contain the spectral
singularities. On the other hand, the multiple eigenvalues are the zeros of
the entire function $\frac{dF(\lambda)}{d\lambda},$ where $F(\lambda)$ is
defined in (3). Since the entire function has a finite number of roots on the
bounded sets the number of the spectral singularities of $H$ is finite
\end{proof}

Now using Theorems 4, 5, Corollary 2, Proposition 2 and the following
equalities we prove the main results of the paper. By Theorem 4 $\lambda
_{n,j}$ satisfies (43) for $j=1,2$ \ and $t\in\lbrack t_{n},\rho].$ Using it
in (45) and (46) we obtain
\begin{equation}
(-C(\lambda_{n,1}(t),t)-4\pi nt-\sqrt{D(\lambda_{n,1}(t),t)})u_{n,1}%
(t)=B(\lambda_{n,1}(t),t)v_{n,1}(t),
\end{equation}%
\begin{equation}
(-C(\lambda_{n,2}(t),t)-4\pi nt+\sqrt{D(\lambda_{n,2}(t),t)})u_{n,2}%
(t)=B(\lambda_{n,2}(t),t)v_{n,2}(t),
\end{equation}%
\begin{equation}
(C(\lambda_{n,1}(t),t)+4\pi nt-\sqrt{D(\lambda_{n,1}(t),t)})v_{n,1}%
(t)=B^{\prime}(\lambda_{n,1}(t),t)u_{n,1}(t),
\end{equation}%
\begin{equation}
(C(\lambda_{n,2}(t),t)+4\pi nt+\sqrt{D(\lambda_{n,2}(t),t)})v_{n,2}%
(t)=B^{\prime}(\lambda_{n,2}(t),t)u_{n,2}(t)
\end{equation}
for $t\in\lbrack t_{n},\rho].$

Since the boundary condition (2) is self-adjoint we have $(H_{t}(q))^{\ast}=$
$H_{t}(\overline{q}).$ Therefore, all formulas and theorems obtained for
$H_{t}$ are true for $H_{t}^{\ast}$ if we replace $a$ and $b$ by $\overline
{b}$ and $\overline{a}$ respectively. For instance, (24) and (25) hold for the
operator $H_{t}^{\ast}$ and hence we have
\begin{equation}
\Psi_{n,j,t}^{\ast}(x)=u_{n,j}^{\ast}(t)e^{i(2\pi n+t)x}+v_{n,j}^{\ast
}(t)e^{i(-2\pi n+t)x}+h_{n,j,t}^{\ast}(x),
\end{equation}%
\begin{equation}
(h_{n,j,t}^{\ast},e^{i(\pm2\pi n+t)x})=0,\text{ }\left\Vert h_{n,j,t}^{\ast
}\right\Vert =O(\frac{1}{n}),\text{ }\left\vert u_{n,j}^{\ast}(t)\right\vert
^{2}+\left\vert v_{n,j}^{\ast}(t)\right\vert ^{2}=1+O(\frac{1}{n^{2}}).
\end{equation}
Note that if (105) holds then $t_{n}=0,$ that is, (106)-(109) are satisfied
for $t\in\lbrack0,\rho].$ For $\mid n\mid>N$ it follows from (24), (25), (110)
and (111) that%
\begin{equation}
(\Psi_{n,j,t},\Psi_{n,j,t}^{\ast})=u_{n,j}(t)\overline{u_{n,j}^{\ast}%
(t)}+v_{n,j}(t)\overline{v_{n,j}^{\ast}(t)}+O(n^{-1}).
\end{equation}
Moreover, by Theorem 4, $\lambda_{n,1}(t)\neq\lambda_{n,2}(t)$ for
$t\in\lbrack t_{n},\rho]$ which imply that%
\begin{equation}
(\Psi_{n,2,t},\Psi_{n,1,t}^{\ast})=u_{n,2}(t)\overline{u_{n,1}^{\ast}%
(t)}+v_{n,2}(t)\overline{v_{n,1}^{\ast}(t)}+O(n^{-1})=0.
\end{equation}

\begin{theorem}
(Main Result) \ If $\mid a\mid=\mid b\mid$ and $\alpha=\pi^{-1}\arg(ab),$ then:

$(a)$ The operator $H$ has no the spectral singularity at infinity and is an
asymptotically spectral operator if and only if (105) holds.

$(b)$ Let $\alpha$ be a rational number, that is, $\alpha=\frac{m}{q}$ where
$m$ and $q$ are irreducible integers. The operator $H$ has no the spectral
singularity at infinity and is an asymptotically spectral operator if and only
if $m$ is an even integer.

Let $\alpha$ be an irrational number. Then $H$ has the spectral singularity at
infinity and is not an asymptotically spectral operator if and only if there
exists a sequence of pairs $\{(q_{k},p_{k})\}\subset\mathbb{N}^{2}$ such
that$\ $%
\begin{equation}
\mid\alpha-\frac{2p_{k}-1}{q_{k}}\mid=o(\frac{1}{q_{k}}),
\end{equation}
where $2p_{k}-1$ and $q_{k}$ are irreducible integers.
\end{theorem}

It is clear that $(b)$ follows from $(a).$ The sufficiency of Theorem 6$(a)$
follows from Remark 4 and the following.

\begin{lemma}
If $\mid a\mid=\mid b\mid$ and (105) holds, then (103) is satisfied.
\end{lemma}

\begin{proof}
If $t\in\lbrack n^{-3},\rho],$ then by (79), (48) and (59) the coefficient of
$u_{n,1}(t)$ in (106) is essentially greater than the coefficient of
$v_{n,1}(t)$. Therefore from (24) and (25) we get
\[
\Psi_{n,1,t}(x)=e^{-i(2\pi n+t)x}+O(n^{-1}).
\]
In the same way we obtain that $\Psi_{n,1,t}^{\ast}$ satisfies the same
formula and hence the equality
\begin{equation}
(\Psi_{n,j,t},\Psi_{n,j,t}^{\ast})=1+O(n^{-1})
\end{equation}
holds uniformly for $t\in\lbrack n^{-3},\rho]$ and $j=1.$

Now suppose that $t\in\lbrack0,n^{-3}].$ First consider the case $nt\geq
\mid\alpha_{n}\mid,$ where $\alpha_{n}$ is defined in Lemma 1. Then using
(79), (85), (87) and taking into account that $\mid\alpha_{n}\mid=\mid
\beta_{n}\mid$ when $\mid a\mid=\mid b\mid$ from (106) we get $\mid
u_{n,1}(t)\mid<\frac{1}{6}\mid v_{n,1}(t)\mid.$ This with (25) gives $\mid
u_{n,1}(t)\mid<\frac{1}{5},$ $\mid v_{n,1}(t)\mid>\frac{4}{5}.$ Similarly,
$\mid u_{n,1}^{\ast}(t)\mid<\frac{1}{5}$ and $\mid v_{n,1}^{\ast}(t)\mid
>\frac{4}{5}.$ Therefore, by (112) we have
\begin{equation}
\mid(\Psi_{n,j,t},\Psi_{n,j,t}^{\ast})\mid>\frac{1}{2}%
\end{equation}
for $j=1.$ In the same way we prove (115) and (116) for $j=2$.

Now, consider the case \ $nt<\mid\alpha_{n}\mid.$ Using (79) and (87) one can
easily see that if $nt=o(\mid\alpha_{n}\mid)$ then both of the number
\begin{equation}
-C(\lambda_{n,1},t)-4\pi nt-\sqrt{D(\lambda_{n,1},t)},\text{ }C(\lambda
_{n,1},t)+4\pi nt-\sqrt{D(\lambda_{n,1},t)},
\end{equation}
if $nt\sim\beta_{n}$ \ then at least one of the numbers in (117) are of order
$\alpha_{n}.$ Hence (106) and (108) imply that $u_{n,1}(t)\sim v_{n,1}(t).$ In
the same way we obtain $u_{n,2}(t)\sim v_{n,2}(t).$ Thus, by (25)%
\begin{equation}
u_{n,1}(t)\sim v_{n,1}(t)\text{ }\sim u_{n,2}(t)\sim v_{n,2}(t)\sim1.
\end{equation}
Similarly
\begin{equation}
u_{n,1}^{\ast}(t)\sim v_{n,1}^{\ast}(t)\text{ }\sim u_{n,2}^{\ast}(t)\sim
v_{n,2}^{\ast}(t)\sim1.
\end{equation}

Using (118), (79) and (87) from (108) and (109) we obtain that
\begin{equation}
\frac{u_{n,1}}{v_{n,1}}=\frac{C(\lambda_{n,1},t)+4\pi nt-\sqrt{D(\lambda
_{n,1},t)}}{B^{\prime}(\lambda_{n,1}(t),t)}=\frac{4\pi nt[1]-\sqrt{4\pi
nt[1]+\alpha_{n}\beta_{n}[1]}}{\alpha_{n}}[1],
\end{equation}%
\begin{equation}
\frac{u_{n,2}}{v_{n,2}}=\frac{C(\lambda_{n,2},t)+4\pi nt+\sqrt{D(\lambda
_{n,2},t)}}{B^{\prime}(\lambda_{n,2}(t),t)}=\frac{4\pi nt[1]+\sqrt{4\pi
nt[1]+\alpha_{n}\beta_{n}[1]}}{\alpha_{n}}[1],
\end{equation}
where, for brevity $1+O(n^{-2})$ is denoted by $[1].$ On the other hand using
(112) and (113) and taking into account (118) and (119) we get
\begin{equation}
(\Psi_{n,1,t},\Psi_{n,1,t}^{\ast})=v_{n,1}(t)\overline{v_{n,1}^{\ast}%
(t)}(1+\frac{u_{n,1}(t)\overline{u_{n,1}^{\ast}(t)}}{v_{n,1}(t)\overline
{v_{n,1}^{\ast}(t)}})+O(n^{-1})=
\end{equation}%
\[
v_{n,1}(t)\overline{v_{n,1}^{\ast}(t)}(1-\frac{u_{n,1}(t)v_{n,2}(t)}%
{v_{n,1}(t)u_{n,2}(t)})+O(n^{-1})).
\]
This with (120) and (121) implies that
\begin{equation}
(\Psi_{n,1,t},\Psi_{n,1,t}^{\ast})=v_{n,1}(t)\overline{v_{n,1}^{\ast}%
(t)}\left(  1-\frac{4\pi nt[1]-\sqrt{(4\pi nt)^{2}[1]+\alpha_{n}\beta_{n}[1]}%
}{4\pi nt[1]+\sqrt{(4\pi nt)^{2}[1]+\alpha_{n}\beta_{n}[1]}}[1]\right)  =
\end{equation}%
\[
v_{n,1}(t)\overline{v_{n,1}^{\ast}(t)}\left(  \frac{2\sqrt{(4\pi
nt)^{2}[1]+\alpha_{n}\beta_{n}[1]}}{4\pi nt[1]+\sqrt{(4\pi nt)^{2}%
[1]+\alpha_{n}\beta_{n}[1]}}[1]\right)
\]
If $(4\pi nt)^{2}=o(\alpha_{n}\beta_{n})$ then the last fraction is $2+o(1)$
and hence (103) holds.

It remains to consider the case $(4\pi nt)^{2}\sim(\alpha_{n}\beta_{n}).$ If
(105) holds then we have inequality (96). Therefore we have $(4\pi
nt)^{2}[1]+\alpha_{n}\beta_{n}[1]\sim\alpha_{n}\beta_{n}.$ Moreover, one can
easily verify that
\[
4\pi nt[1]+\sqrt{(4\pi nt)^{2}[1]+\alpha_{n}\beta_{n}[1]}=4\pi nt(1+\sqrt
{1+(4\pi nt)^{-2}\alpha_{n}\beta_{n}}+o(1))\sim(\alpha_{n}\beta_{n}).
\]
Using these relations in (123) we get (103) in this case. Thus (103) for
$t\in\lbrack0,\rho]$ is proved. In the same way, by using Remark 3, we prove
(103) for $t\in\lbrack\pi-\rho,\pi]$ and it follows from (5) for $t\in
\lbrack\rho,\pi-\rho]$
\end{proof}

The proof of the necessity of Theorem 6$(a)$ follows from Proposition 2 and
the following

\begin{lemma}
If $\mid a\mid=\mid b\mid$ and
\begin{equation}
\text{ }\inf_{q,p\in\mathbb{N}}\{\mid q\alpha-(2p-1)\mid\}=0,
\end{equation}
where $\alpha=\pi^{-1}\arg(ab),$ then there exists a sequence of pairs
$\{(n_{k},t_{k})\}$ satisfying (104), where $n_{k}\in\mathbb{Z},$ $t_{k}%
\in\lbrack0,\pi]$ and $\lambda_{n_{k}}(t_{k})$ is a simple eigenvalue
\end{lemma}

\begin{proof}
If (124) holds then there exists a sequence of pairs $\{(q_{k},p_{k})\}$ such
that $q_{k}\alpha-(2p_{k}-1)\rightarrow0.$ First suppose that the sequence
$\{q_{k}\}$ contains infinite number of even number. Then one can easily
verify that there exists a sequence $\{n_{k}\}$ satisfying
\begin{equation}
\operatorname{Im}((ab)^{2n_{k}})=o((ab)^{2n_{k}})
\end{equation}
and%
\begin{equation}
\lim_{k\rightarrow\infty}sgn(\operatorname{Re}((ab)^{2n_{k}}))=-1.
\end{equation}
By Theorem 4, for the sequence $\{t_{n_{k}}\}$ defined by (91) and now, for
simplicity, redenoted by $\{t_{k}\}$ the eigenvalues $\lambda_{n_{k},j}%
(t_{k})$ are simple and the following relations hold%
\[
(4\pi n_{k}t_{k})^{2}=-\operatorname{Re}(\beta_{n_{k}}\alpha_{n_{k}%
})(1+o(1))=-(\beta_{n_{k}}\alpha_{n_{k}})(1+o(1)),\text{ }%
\]%
\begin{equation}
(4\pi n_{k}t_{k})^{2}+(\beta_{n_{k}}\alpha_{n_{k}})=o(\beta_{n_{k}}%
^{2}),\text{ }4\pi n_{k}t_{k}\sim\beta_{n_{k}}\sim\alpha_{n_{k}}.
\end{equation}
This with (87) implies that $\sqrt{D(\lambda_{n_{k},j},t_{k})}=o(\beta_{n_{k}%
})$ for $j=1,2.$ Then by (79) we have
\begin{equation}
C(\lambda_{n_{k},j},t_{k})+4\pi n_{k}t_{k}\pm\sqrt{D(\lambda_{n_{k}j},t_{k}%
)}=4\pi n_{k}t_{k}(1+o(1)).
\end{equation}
Using (61), (127) and (128) in (106) and (107) and taking into account (25) we
obtain
\[
u_{n_{k},1}(t)\sim v_{n_{k},1}(t)\sim u_{n_{k},2}(t)\sim v_{n_{k},2}(t)\sim1,
\]%
\[
\lim_{k\rightarrow\infty}\frac{v_{n_{k},1}(t_{k})}{u_{n_{k},1}(t_{k})}%
=\lim_{k\rightarrow\infty}\frac{v_{n_{k},2}(t_{k})}{u_{n_{k},2}(t_{k})}.
\]
This with (112) and (113) implies that (104) holds. In the same way we prove
(104) when $\{q_{k}\}$ contains infinite number of odd number.
\end{proof}

\begin{remark}
The main result of this paper shows that the asymptotic spectrality of $H$
depends on $\arg(ab),$ while we have proved in [15] that its spectrum depends
on $ab.$
\end{remark}

\end{document}